\documentclass[11pt, a4paper]{article}
\usepackage{tikz}
\usetikzlibrary{calc}
\usepackage[utf8]{inputenc}
\usepackage{amsfonts}
\usepackage{appendix}
\usepackage{mathrsfs}
\usepackage{babel}
\usepackage{amsmath}
\date{}
\usepackage{dsfont}
\usepackage{amssymb}
\usepackage{graphicx}
\usepackage{yfonts}
\usepackage{array}
\usepackage{amsthm}
\usepackage[linkcolor=blue,colorlinks=true]{hyperref}
\usepackage{xcolor}
\newtheorem{Def}{Definition}
\newtheorem{thm}{Theorem}
\newtheorem{prop}{Proposition}
\newtheorem{lem}{Lemma}
\newtheorem{cor}{Corollary}
\newtheorem{rem}{Remark}

\newcolumntype{R}[1]{>{\raggedleft\arraybackslash }b{#1}}
\newcolumntype{L}[1]{>{\raggedright\arraybackslash }b{#1}}
\newcolumntype{C}[1]{>{\centering\arraybackslash }b{#1}}
\newcommand{\R}{\mathbb{R}}
\newcommand{\Z}{\mathbb{Z}}

\newcommand{\e}{\eta}

\numberwithin{equation}{section}

\usepackage{geometry} 
\title{Extinction and survival in inherited sterility}

\date{}

\usepackage{authblk}

\newsavebox\affbox
\author[1]{\text{Sonia Velasco}}
\affil[1]{ Instituto de Matemática Pura e Aplicada.}

\begin{document}

\maketitle

\begin{abstract} 
We introduce an interacting particle system which models the inherited sterility method. Individuals evolve on $\Z^d$ according to a contact process with parameter $\lambda>0$. With probability $p\in [0,1]$ an offspring is fertile and can give birth, on neighboring sites, to other individuals at rate $\lambda$. With probability $1-p$, an offspring is sterile and blocks the site it sits on until it dies. The goal is to prove that at fixed $\lambda$, the system survives for large enough $p$ and dies out for small enough $p$. The model is not attractive, since an increase of fertile individuals potentially causes that of sterile ones. However, thanks to a comparison argument with attractive models, we are able to answer our question. 
\end{abstract}


\section{Introduction}
In this paper, we introduce an interacting particle system, suggested to us by Rinaldo Schinazi \cite{Rinaldo}, to model the "Inherited Sterility" (IS) method. Developed in the mid-20th century, see \cite{ISinLepidoptera}, this pest control technique, particularly effective against crop-damaging invasive species, is an adaptation of the Sterile Insect Technique (SIT) introduced by E. Knipling in the 1950s to eliminate New World screw worms. In SIT, sterilized (thanks to gamma rays) males are released to mate with fertile females, leading to no offspring and eventual population decline. However, some species, like Lepidoptera, require a high level of radiation to be completely sterilized, reducing the competitiveness of treated males. To counterbalance this effect, in the IS method, the species is partially sterilized so that it produces a certain proportion of sterile offspring, and another of fertile ones. The fertile offspring themselves then have a certain chance of giving birth to fertile or sterile individuals and so on. We refer to the reference book \cite{TIS} for a detailed list of trials and programs regarding the SIT and to its Chapter 2.4 for Inherited Sterility.

For a mathematical analysis of the SIT, an interacting particle system is introduced in \cite{CPRS} as a toy model. In this paper, a phase transition result for survival or extinction is established at the microscopic level and in infinite volume. The macroscopic, out-of-equilibrium behavior of this system is explored in \cite{KMS} in infinite volume and more recently in \cite{MSV} at equilibrium in finite volume with slow reservoirs. Another interacting particle system for SIT is introduced in \cite{Durrett}, also yielding a phase transition result. Both \cite{CPRS} and \cite{Durrett} rely on the monotonicity of the dynamics: fertile individuals increase the likelihood of survival. However, this does not apply to the IS method, as more fertile individuals can lead to more sterile ones over time. This notable fact makes the mathematical analysis of our IS model challenging. 

Our model is a variation of the contact process on $\Z^d$. Recall that in the contact process, sites are either empty (in state $0$), or occupied (in state $1$) by an individual which can give birth to another individual on a neighbouring empty site at a given rate $\lambda>0$, or die at rate $1$. In our model, sites can be empty (in state $0$), occupied by a fertile individual (in state $1$), or occupied by a sterile individual (in state $-1$). A fertile individual can give birth to another individual on a neighbouring empty site at a given rate $\lambda>0$. This new individual is fertile with probability $p\in [0,1]$, and sterile with probability $1-p$. Sterile individuals cannot reproduce, and individuals, whether fertile or sterile, die at rate $1$. We refer to this process as an $IS(\lambda,p)$ process and our goal is to establish conditions on $\lambda$ and $p$ under which the population of fertile individuals survives or becomes extinct.

Recall that for the contact process, there is a critical value $\lambda_c(d)\in (0,\infty)$ such that for $\lambda\leq \lambda_c(d)$, the population dies out, and for $\lambda>\lambda_c(d)$ it survives. We refer to \cite[Chapter 6]{IPS} and \cite{Liggett2} for detailed surveys on the contact process. An essential ingredient in the proof of this phase transition result is the monotonicity of the contact process in the parameter $\lambda$, that is to say, that given a contact process $(\zeta^1_t)_{t\geq 0}$ with parameter $\lambda_1$, and a contact process $(\zeta^2_t)_{t\geq 0}$ with parameter $\lambda_2$, if $\lambda_1\leq \lambda_2$, $\zeta^2$ stochastically dominates $\zeta^1$. In particular, if $\zeta^2$ becomes extinct, so does $\zeta^1$ and if $\zeta^1$ survives, so does $\zeta^2$. Unfortunately, in our model for inherited sterility, there is no monotonicity in any of the parameters of the model (see Proposition \ref{Nomonot}) so we are not able to establish a threshold value in $p$ or $\lambda$ under which the process becomes extinct and above which it survives. However, we manage to prove the following.\medskip 

Fix $d\geq 1$, and recall that $\lambda_c(d)$ is the critical value for the contact process.
    \begin{itemize}
        \item[(i)] If $\lambda>0$ and $p\in [0,1]$ are such that $\lambda p \leq \lambda_c(d)$, an $IS(\lambda,p)$ process on $\Z^d$ almost surely becomes extinct.
        \item[(ii)] If $\lambda>\lambda_c(d)$, there exists a $\check{p}(\lambda)\in [\lambda_c(d)/\lambda,1)$ such that for any $p\geq \check{p}(\lambda)$, an $IS(\lambda,p)$ process on $\Z^d$ survives.
    \end{itemize}
The strategy pursued is the following: to prove $(i)$ we show  that our process is stochastically dominated by a basic contact process which becomes extinct when $\lambda p$ is small enough. To prove $(ii)$, we introduce a contact process with a dynamic random environment which survives when $\lambda$ and $p$ are large enough, and which we refer to as a $Spont(\lambda,p)$ process. This process is similar to the one studied in \cite{LinkerRemenik} where edges randomly become blocked and prevent sites from reproducing, except that in our setting, it is sites rather than edges which are randomly blocked.  We show that an $IS(\lambda,p)$ process stochastically dominates a $Spont(\lambda,p)$ process, and therefore survives too. The main part of the work then relies on proving the survival of $Spont(\lambda,p)$. For that, in the spirit of \cite{Coex} and \cite{CPRS}, we compare its graphical representation to oriented percolation, and use a renormalization argument.

Our proof simplifies the renormalization strategy pursued in \cite{Coex} and \cite{CPRS}, let us briefly explain why. The renormalization argument involves defining a specific event within a space-time box in the graphical representation of the process, which enables the spread of the population. By a comparison argument between the graphical representation and an oriented percolation, if one proves that tuning the size of the space-time box, the event happens with large enough probability, clustering in the underlying oriented percolation graph will imply survival of the population. The crux is thus the choice of the event and the proof that it happens with large enough probability. In \cite{Coex} and \cite{CPRS}, the proof is done in one dimension, and extended to any dimension by an embedding argument, which fails here due to lack of monotonicity. Our renormalization argument must therefore directly be performed in arbitrary dimension. For that, we bypass a crucial step needed in $\cite{Coex}$ and $\cite{CPRS}$, which guarantees the presence of enough $1$'s after a certain time in the space-time block, but uses properties of the contact process in one dimension. Instead, we rely on a duality argument, valid in any dimension, which allows us to subdivide the event into fewer events, whose probability can be made arbitrarily large.
\medskip

The paper is organized as follows. In Section \ref{DandR}, we introduce the models and state all the results. In Section \ref{graphrepr} we define the graphical representation associated to each model and prove the stochastic dominations. In section \ref{PMR}, we prove the survival of the $Spont(\lambda,p)$ process.

\section{Definitions and results}\label{DandR}
\subsection{The inherited sterility model and main result}
For $d\geq 1$, introduce the state space $\Omega = \{-1,0,1\}^{\Z^d}$, so that for $\e\in \Omega$ and $x\in \Z^d$, $\e(x)$ is the state of site $x$ in $\e$. 
We say that 
\begin{equation} 
   \e(x)= \left\{
    \begin{array}{ll}
        1,~ ~\text{if there is a fertile individual at site } x,\\ 
        -1,~ ~\text{if there is a sterile individual at site } x,\\
        0,~ ~\text{if site $x$ is empty}.
        \end{array}
\right.
\end{equation}
Two sites $x$ and $y$ are nearest neighbours in $\Z^d$ if $\|x-y\|_1=1$ and we write $x\sim y$. 

Introduce $\e_1,\e_{-1}, \e_{0}\in \{0,1\}^{\Z^d}$ as follows:
\begin{equation}
    \e_1(x) = \mathds{1}_{\e(x)=1},~ ~~ \e_{-1}(x) = \mathds{1}_{\e(x)=-1},~ ~~ \e_{0}(x) = \mathds{1}_{\e(x)=0}.
\end{equation}
For $x\in \Z^d$ and $\e\in \Omega$, denote by $n_1(x,\e) = \underset{y\sim x}{\sum}\e_1(y) $ the number of neighbours of $x$ in state $1$ in $\e$.

\begin{Def}\label{DefinitionofIS}
    The inherited sterility process with birth rate $\lambda>0$ and fertility probability $p\in [0,1]$, that we will refer to as $IS(\lambda,p)$, is the Markov process $(\e_t)_{t\geq 0}$ on the state space $\Omega$, whose transition rates at $x\in \Z^d$ for a current configuration $\e$ are given by:
    \begin{equation}\label{ratesIS}
        1,-1 \rightarrow 0: \mathrm{at~ rate~ }1,~ ~ 0 \rightarrow 1: \mathrm{at~ rate~ }\lambda p n_1(x,\e),~ ~ 0 \rightarrow -1: \mathrm{at~ rate~ }\lambda (1-p) n_1(x,\e).
    \end{equation}
\end{Def}
For $\e\in \Omega$, $x\in \Z^d$ and $i\in \{-1,0,1\}$, denote by $\sigma^{i,x}\e$ the configuration obtained from $\e$ after flipping the state of $x$ to $i$:
\begin{equation} 
   \sigma^{i,x}\e(y)= \left\{
    \begin{array}{ll}
        i,~ ~\text{if}~ y=x\\ 
        \e(y),~ ~\text{otherwise}.
        \end{array}
\right.
\end{equation}
The infinitesimal generator of an $IS(\lambda,p)$ process is given by: for any cylinder function $f$ on $\Omega$ and configuration $\e\in \Omega$,
\begin{equation}\label{genIS}
        \begin{split}
        \mathcal{L}f(\e) &= \sum_{x\in \Z^d} \sum_{i\in \{-1,0,1\}} c(x,\e,i)\big[f(\sigma^{i,x}\e) -f(\e)\big],
    \end{split}
\end{equation}
where $c(x,\e,i)$ is the transition rate to go from state $\e(x)$ to $i$. The rates are given by:
\begin{equation}\label{ratesgen}
\begin{split}
        &c(x,\e,0) = 1,~~~\text{  if  } \e(x)\in \{-1,1\},\\
        &c(x,\e,1) = \lambda p n_1(x,\e),~~~\text{  if  } \e(x)=0,\\
        &c(x,\e,-1)= \lambda (1-p) n_1(x,\e),~~~\text{  if  } \e(x)=0.
\end{split}
\end{equation}
Since all the rates in \eqref{ratesgen} are bounded, by \cite[Chapter 1, Theorem 3.9]{IPS}, there exists a unique Markov process whose dynamics is induced by the infinitesimal generator \eqref{genIS}.\medskip

For $\e\in \Omega$, we will denote by $\mathbb{P}_{\e}^{\lambda,p}$ the probability measure on the space of continuous time trajectories on $\Omega$ induced by $(\e_t)_{t\geq 0}$ when $\e_0=\e$. We also denote by
\begin{equation}\label{setoffertile}
    A(\e) =  \{x\in \Z^d,~ \e(x)=1\}.
\end{equation}
An $IS(\lambda,p)$ process $(\e_t)_{t\geq 0}$ is said to \textit{survive} if,
\begin{equation}
    \mathbb{P}_{\{0\}}^{\lambda,p}\Big( \forall t> 0,~  A(\e_t)\neq\emptyset \Big)>0,
\end{equation}
where, by abuse of notation, $\{0\}$ is the configuration containing a $1$ at site $0$ and $0$'s everywhere else. The process is said to \textit{become extinct} otherwise.

\begin{thm}\label{Mainresult}
Fix $d\geq 1$, and recall that $\lambda_c(d)$ is the critical value for the contact process.
    \begin{itemize}
        \item[(i)] If $\lambda>0$ and $p\in [0,1]$ are such that $\lambda p \leq \lambda_c(d)$, an $IS(\lambda,p)$ process on $\Z^d$ almost surely becomes extinct.
        \item[(ii)] If $\lambda>\lambda_c(d)$, there exists a $\check{p}(\lambda)\in [\lambda_c(d)/\lambda,1)$ such that for any $p\geq \check{p}(\lambda)$, an $IS(\lambda,p)$ process on $\Z^d$ survives.
    \end{itemize}
\end{thm}
In the rest of the paper, if $X$ is an ordered set, given two configurations $\xi_1$ and $\xi_2$ in $X^{\Z^d}$, we say that $\xi_1\leq \xi_2$ if for any $x\in \Z^d$, $\xi_1(x)\leq \xi_2(x)$.\medskip

A convenient tool for the proof of extinction and survival in the context of non conservative particle systems is monotonicity, defined as follows:

\begin{Def}
    Consider $X$ an ordered set. A process $(\zeta_t)_{t\geq 0}$ with values in $X^{\Z^d}$ and whose dynamics is parameterized by a certain value $q$ is said to be monotone in $q$ if, when $q_1\leq q_2$, one can couple $(\zeta^{(1)}_t)_{t\geq 0}$ with dynamics parameter $q_1$, and  $(\zeta^{(2)}_t)_{t\geq 0}$ with dynamics parameter $q_2$, in such a way that
    $$\zeta^{(1)}_0 \leq \zeta^{(2)}_0 \Rightarrow \zeta^{(1)}_t \leq \zeta^{(2)}_t~ a.s.~ \text{for all } t>0. $$
\end{Def}
For our model, there is no monotonicity in $p$:
\begin{prop}\label{Nomonot}
    For any ordering of $\{-1,0,1\}$ and any $\lambda>0$, an $IS$ process on $\Z^d$ with birth rate $\lambda$ is not monotone in the parameter $p$.
\end{prop}
The proof of Proposition \ref{Nomonot} is done in Section \ref{loa}.
\begin{rem}
    It follows that in Theorem \ref{Mainresult}, one cannot rely on a monotonicity argument to prove that the phase transition in $p$ is sharp in the sense that: for $\lambda>\lambda_c(d)$, there is a critical parameter $p_c(\lambda)\in [\lambda_c(d)/\lambda,1)$ such that for any $p<p_c(\lambda)$, an $IS(\lambda,p)$ process becomes extinct and for any $p>p_c(\lambda)$, an $IS(\lambda,p)$ process survives. 
\end{rem}

\subsection{Two other processes}\label{TIP}
\subsubsection*{The contact process}
Recall that the contact process on $\Z^d$ with parameter $\lambda$ is an interacting particle system on the state space $\{0,1\}^{\Z^d}$, whose transition rates at $x$ for a current configuration $\zeta$ are given by:
\begin{equation}
    0\rightarrow 1: ~\text{at rate  } \lambda n_1(x,\zeta),~ ~ \text{and  } 1\rightarrow 0:~\text{at rate  }1.
\end{equation}
The contact process on $\Z^d$ exhibits a phase transition in the parameter $\lambda$ (we refer to \cite[Part 1, section 2]{Liggett2}) : there is a $\lambda_c(d)\in (0,\infty)$ such that for any $\lambda \leq \lambda_c(d)$, the contact process with parameter $\lambda$ almost surely reaches the empty configuration (extinction), and for $\lambda>\lambda_c(d)$, with strictly positive probability, the process never reaches the empty configuration (survival). We refer to \cite{IPS} and \cite{Liggett2} for detailed reviews on the contact process.\medskip

For $\zeta\in \Omega$, we denote by $\textbf{P}_{\zeta}^{\lambda}$ the probability measure on the space of continuous time trajectories on $\Omega$ induced by $(\zeta_t)_{t\geq 0}$, when $\zeta_0=\zeta$.

\begin{rem}
    Note that if $p=1$, an $IS(\lambda,p)$ process starting from a configuration in $\{0,1\}^{\Z^d}$ evolves according to a contact process on $\Z^d$ with parameter $\lambda$.
\end{rem}
\begin{thm}\label{Compcontact}
    For any $(\lambda,p)\in (0,\infty)\times[0,1]$, and $(\e_0,\zeta_0)\in \Omega \times \{0,1\}^{\mathbb{Z}^d}$ such that $\e_0\leq \zeta_0$, there exists a coupling $(\e_t,\zeta_t)_{t\geq 0}$, on $\Omega \times \{0,1\}^{\mathbb{Z}^d}$, such that $(\e_t)_{t\geq 0}$ is an $IS(\lambda,p)$ process starting from $\e_0$, $(\zeta_t)_{t\geq 0}$ a contact process with birth rate $\lambda p$ starting from $\zeta_0$, and almost surely,
    \begin{equation*}
        \forall t \geq 0,~~\e_t\leq \zeta_t,
    \end{equation*}
    for the order $-1<0<1$.
\end{thm}
The proof of Theorem \ref{Compcontact}  is done using the basic coupling on the graphical representation and we refer to Section \ref{COUPLINGS} for more details.

\subsubsection*{A contact process with dynamic random environment}
\begin{Def}
    For $\lambda>0$ and $p\in [0,1]$, the $Spont(\lambda,p)$ process is the Markovian process $(\xi_t)_{t\geq 0}$ on the state space $\Omega$ whose transition rates at site $x\in \Z^d$ for a current configuration $\xi$ are given by:
     \begin{equation}
        0 \rightarrow 1:~\mathrm{at~ rate~ }  \lambda p n_1(x,\xi),~~0\rightarrow -1:~\mathrm{at~ rate~ } 2d\lambda (1-p) ~~ \mathrm{and}~~1,-1\rightarrow 0:~\mathrm{at~ rate~ } 1.
    \end{equation}
\end{Def}
In other words, the dynamics is that of a contact process with parameter $\lambda p$, where empty sites become randomly blocked, i.e., no sites in state $1$ can reproduce on neighbouring blocked sites, before these blocked sites flip back to $0$.

The infinitesimal generator of a $Spont(\lambda,p)$ process is given by: for any cylinder function $f$ on $\Omega$ and configuration $\xi\in \Omega$,
\begin{equation}\label{genspont}
        \begin{split}
        \mathbb{L}f(\xi) &= \sum_{x\in \Z^d} \sum_{i\in \{-1,0,1\}} c_{spont}(x,\xi,i)\big[f(\sigma^{i,x}\xi) -f(\xi)\big],
    \end{split}
\end{equation}
where $c_{spont}(x,\xi,i)$ is the transition rate to go from state $\xi(x)$ to $i$. The transition rates are given by:
\begin{equation}\label{ratesgenspont}
\begin{split}
        &c_{spont}(x,\xi,0) = 1,\text{  if  } \xi(x)\in \{-1,1\},\\
        &c_{spont}(x,\xi,1) = \lambda p n_1(x,\xi),\text{  if  } \xi(x)=0,\\
        &c_{spont}(x,\xi,-1)= 2d\lambda (1-p),\text{  if  } \xi(x)=0.
\end{split}
\end{equation}
Since all the rates in \eqref{ratesgenspont} are bounded, by \cite[Chapter 1, Theorem 3.9]{IPS} there exists a unique Markov process whose dynamics is induced by the infinitesimal generator \eqref{genspont}.\medskip

For $\xi\in \Omega$, we will denote by $\Tilde{\mathbb{P}}_{\xi}^{\lambda,p}$ the probability measure on the space of continuous time trajectories on $\Omega$ induced by $(\e_t)_{t\geq 0}$, when $\xi_0=\xi$.\medskip

The following result tells us that a $Spont(\lambda,p)$ process is stochatically dominated by an $IS(\lambda,p)$ process:
\begin{thm}\label{CouplingISwithSpont}
    For any $(\lambda,p)\in (0,\infty)\times[0,1]$, and $(\xi_0,\e_0)\in \Omega^2$ such that $\e_0\leq \zeta_0$, there exists a coupling $(\xi_t,\eta_t)_{t\geq 0}$, on $\Omega^2$ such that $(\xi_t)_{t\geq 0}$ is an $Spont(\lambda,p)$ process starting from $\xi_0$, $(\eta_t)_{t\geq 0}$ an $IS(\lambda,p)$ process starting from $\eta_0$ such that almost surely,
    \begin{equation*}
        \forall t\geq 0,~~~\xi_t\leq \eta_t
    \end{equation*}
    for the order $-1<0<1$.
\end{thm}
The proof of Theorem \ref{CouplingISwithSpont} is done using the basic coupling on the graphical representation and we refer to Section \ref{COUPLINGS} for more details.\medskip

The $Spont$ process satisfies the following.
\begin{prop}\label{MonotonicityofSpont}
Fix $\lambda>0$ and $p_1\leq p_2$ in $[0,1]$. There exists a coupling $(\xi_t^{(1)},\xi_t^{(2)})_{t\geq 0}$ on $\Omega^2$ such that $(\xi_t^{(1)})_{t\geq 0}$, resp. $(\xi_t^{(2)})_{t\geq 0}$ is a $Spont(\lambda,p_1)$,  resp. $Spont(\lambda,p_2)$, and almost surely, one has
\begin{equation}\label{monot}
    \forall t \geq 0,~~~\xi_0^{(1)}\leq \xi_0^{(2)} \Rightarrow~ ~~  \xi_t^{(1)}\leq \xi_t^{(2)},
\end{equation}
for the order $-1<0<1$.
\end{prop}
The proof of Proposition \ref{MonotonicityofSpont} is postponed to Section \ref{loa} and relies on the basic coupling.

\begin{thm}{Phase transition for the $Spont$ process.}\label{Phaseforspont}

\noindent Fix $\lambda>\lambda_c(d)$. The $Spont$ process with birth rate $\lambda$ exhibits a non trivial phase transition in the parameter $p$: there exists $p_c^{spont}(\lambda)\in [\lambda_c(d)/\lambda,1)$, such that
\begin{itemize}
    \item [(i)] For any $p<p_c^{spont}(\lambda)$, the process $Spont(\lambda,p)$ becomes extinct.
    \item [(ii)] For any $p>p_c^{spont}(\lambda)$, the process $Spont(\lambda,p)$ survives.
\end{itemize}
    
\end{thm}
The proof of Theorem \ref{Phaseforspont} is postponed to Section \ref{PMR} and relies on the fact that for $Spont$, contrary to $IS$, monotonicity holds, as stated in Proposition \ref{MonotonicityofSpont}.

\begin{rem}
    Taking $p_1=p_2$ in Proposition \ref{MonotonicityofSpont}, we get that the $Spont$ process is attractive in $p$. This means that one can couple two $Spon(\lambda,p)$ processes with same parameter $p$, so that if one is smaller than the other, this remains true at any later time.
\end{rem}

\subsection{Proof of Theorem \ref{Mainresult}}
Collecting the results stated in Section \ref{TIP}, we are in position to prove Theorem \ref{Mainresult}.\medskip

\begin{itemize}
    \item [(i)]By Theorem \ref{Compcontact}, we can consider a coupling $(\e_t,\zeta_t)$ between an $IS(\lambda,p)$ process and a contact process with parameter $\lambda p$, both starting from the configuration $\{0\}$, such that
\begin{equation*}
    A(\e_t) \subset A(\zeta_t)~~~a.s.~~~\forall t \geq 0.
\end{equation*}
  Therefore,
\begin{equation}\label{uppper}
    \mathbb{P}_{\{0\}}^{\lambda,p}\Big(\forall t,~  A(\e_t)\neq \emptyset \Big) \leq \textbf{P}_{\{0\}}^{\lambda p}\Big(\forall t,~  A(\zeta_t)\neq \emptyset \Big).
\end{equation}
It follows that for $p\leq \lambda_c(d)/\lambda$, a contact process with parameter $\lambda p$ becomes extinct so the upper bound in \eqref{uppper}  is zero. Hence, the $IS(\lambda,p)$ process become extinct. In particular, this holds for any $p\in [0,1]$, as soon as $\lambda\leq \lambda_c(d)$.

\item[(ii)] As just seen, for $\lambda\leq \lambda_c$, for any $p\in [0,1]$, an $IS(\lambda,p)$ process become extinct. Fix $\lambda >\lambda_c$. By Theorem \ref{CouplingISwithSpont} we can consider a coupling $(\xi_t,\eta_t)$ between a $Spont(\lambda,p)$ process and an $IS(\lambda,p)$ both starting from the configuration $\{0\}$, such that
\begin{equation*}
    A(\xi_t) \subset A(\eta_t)~~~a.s.~~~\forall t \geq 0.
\end{equation*}
Therefore,
\begin{equation*}\label{lowwer}
    \Tilde{\mathbb{P}}_{\{0\}}^{\lambda,p}\Big(\forall t,~  A(\xi_t)\neq \emptyset \Big) \leq \mathbb{P}_{\{0\}}^{\lambda,p}\Big(\forall t,~  A(\e_t)\neq \emptyset \Big).
\end{equation*}
By Theorem \ref{Phaseforspont}, for $p>p_c^{spont}(\lambda)$, the process $Spont(\lambda,p)$ survives so the lower bound in \eqref{lowwer} is strictly positive. It turns out that for any $p>p_c^{spont}(\lambda)$ the $IS(\lambda,p)$ process survives and $p_c(\lambda)\in [\lambda_c(d)/\lambda,1)$. Taking $\check{p}(\lambda) =p_c(\lambda) $, the result follows.
\end{itemize}

\section{Graphical representations and couplings}\label{graphrepr}
The processes introduced in Section \ref{DandR} can be alternatively described by a graphical representation, which gives another way of defining their dynamics, through the use of Poisson point processes. This construction was introduced by Harris, see \cite{Harris}. The advantage of the graphical representation is that it allows to build very natural couplings between processes, and in particular, to prove some monotonicity properties. It also allows to compare the evolution of the set of occupied sites to that of a percolation cluster on an oriented percolation graph. This key feature will be central in the following Section.

\subsection{Graphical representations}

Fix $\lambda>0$ and $p\in [0,1]$. Consider the diagram $\Z^d\times \R_+$. Denote by $E(\Z^d)$ the set of oriented edges of $\Z^d$. To each element $(x,y)\in E(\Z^d) $ we associate the realization of a Poisson point process $(N_1^{x,y})_{(x,y)\in E(\Z^d)}$ of parameter $\lambda p$, as well as that of  a Poisson point process $(N_2^{x,y})_{(x,y)\in E(\Z^d)}$ of parameter $\lambda (1-p).$ Also, consider two families of realizations of Poisson point processes $(U_x)_{x\in \Z^d}$ and $(V_x)_{x\in \Z^d}$ with rate $1$. We suppose that all these Poisson processes are sampled independently. At each time event $t$ of $N_1^{x,y}$, resp. $N_2^{x,y}$, draw an arrow $\overset{1}{\longrightarrow}$, resp. $\overset{-1}{\longrightarrow}$,   in $\mathbb{Z}^d\times \R_+$ from $(x,t)$ to $(y,t)$, as illustrated by the blue, resp. red, arrows in Figure \ref{FigureGRIS}. At each time event $t$ of $U_x$, resp. $V_x$ , place a symbol $\times$, resp. $o$, at $(x,t)$, as illustrated by the black crosses, resp. red dots in Figure \ref{FigureGRIS}. From the arrows, dots and crosses, one can build the $IS(\lambda,p)$ process, the $Spont(\lambda,p)$ process and the contact process with parameter $\lambda$ as follows:
\begin{itemize}
    \item For $IS(\lambda,p)$: at each arrow $\overset{1}{\longrightarrow}$ in $\mathbb{Z}^d\times \R_+$ from $(x,t)$ to $(y,t)$, if $x$ is in state $1$ and $y$ in state $0$, the birth of a fertile individual occurs in $y$, that is, it flips to state $1$, see blue arrows in Figure \ref{FigureGRIS}. At each arrow $\overset{-1}{\longrightarrow}$ in $\mathbb{Z}^d\times \R_+$ from $(x,t)$ to $(y,t)$, if $x$ is in state $1$ and $y$ in state $0$, the birth of a sterile individual occurs in $y$, that is, it flips to state $-1$, see red arrows in Figure \ref{FigureGRIS}. At an event $\times$, resp. $o$, positioned at $(x,t)$, if $x$ was in state $1$, resp. $-1$ it flips to state $0$.
    \item For $Spont(\lambda,p)$: perform the same steps as for the graphical representation of $IS(\lambda,p)$ except that arrow $\overset{-1}{\longrightarrow}$ in $\mathbb{Z}^d\times \R_+$ from $(x,t)$ to $(y,t)$, if $y$ is in state $0$, it flips to state $-1$, regardless of the state of $x$.
    \item For the contact process with parameter $\lambda p$: perform the same steps as for the graphical representation of $IS(\lambda,p)$ and ignore the effects of the arrows $\overset{-1}{\longrightarrow}$.
\end{itemize}
\begin{figure}
        \begin{center}\label{FigureGRIS}
            \includegraphics[scale=0.35]{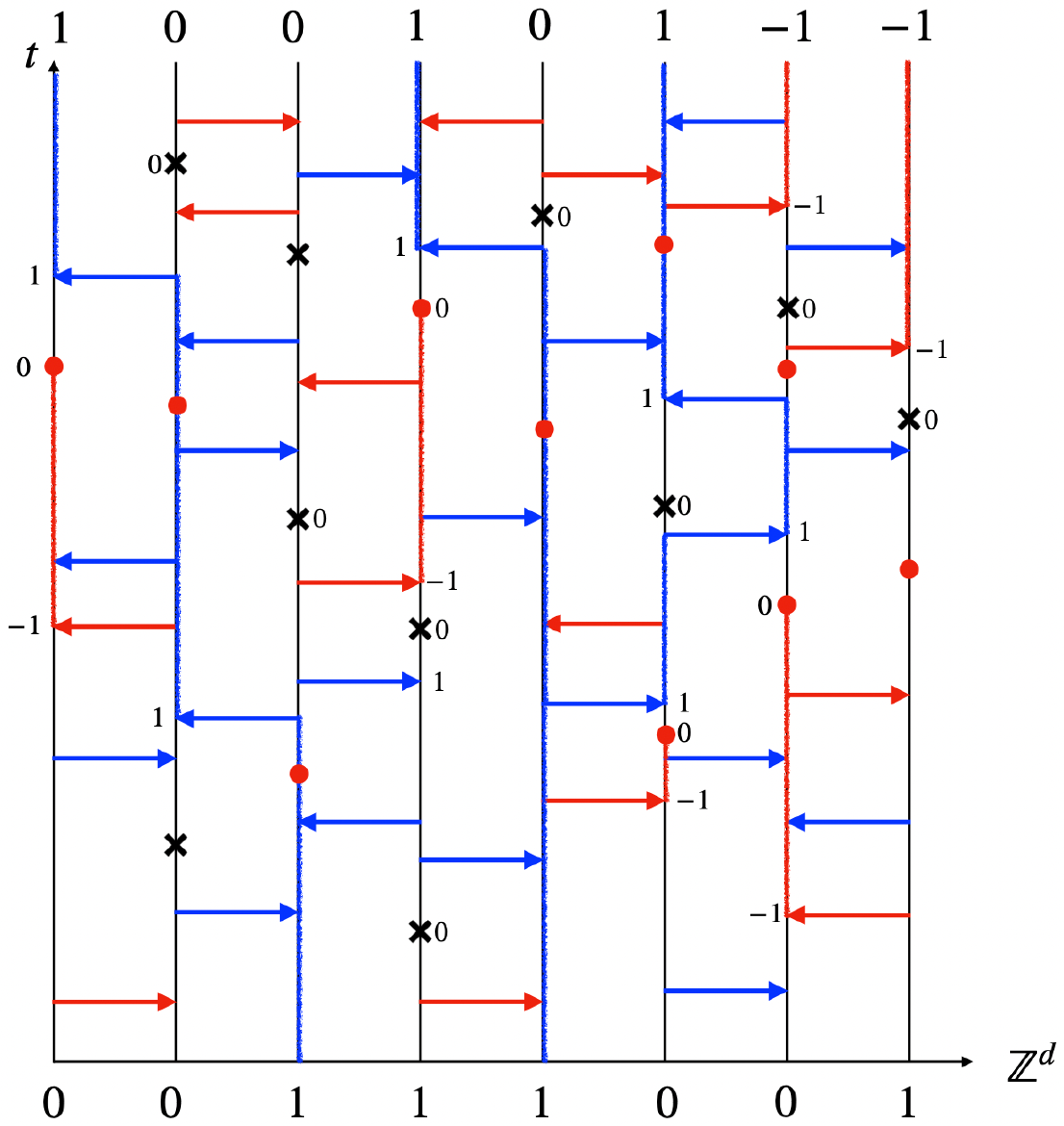}
        \end{center}
        \caption{\textbf{Graphical representation for a one dimensional $IS$ process.} The blue, resp. red arrows, correspond to births of fertile, resp. sterile individuals. The black crosses, resp. red dots, correspond to deaths of fertile, resp. sterile individuals. The blue paths correspond to space-time active paths along which individuals survive. The red segments correspond to blocked sites due to the presence of a sterile individual, that is, sites where fertile individuals cannot be born, until the sterile individual blocking the site dies.}
    \end{figure}

Given a graphical representation, an active path refers to a connected oriented path, moving along the time lines in the increasing direction of time and passing along arrows $\overset{1}{\longrightarrow}$, which crosses neither symbols $\times$ nor space-time points that are in state $-1$. Then, an $IS(\lambda,p)$, resp. $Spont(\lambda,p)$, resp. contact process starting from a configuration $\e$, resp. $\xi$, resp. $\zeta$, can be built from the percolation structure described above, following the indications given by the different space time events. In particular, if the process starts with $1$'s in a given set $\mathcal{A}(\e)$, resp. $\mathcal{A}(\xi)$, resp. $\mathcal{A}(\zeta)$, and $0$'s everywhere else, the set of $1$'s at time $t$ in $(\e_t)_{t\geq 0}$, resp. $(\xi_t)_{t\geq 0}$, resp. $(\zeta_t)_{t\geq 0}$  is given by:
\begin{equation*}
    \mathcal{A}_t(\e)= \Big\{y\in \Z^d,~ \exists x\in \mathcal{A}(\e)~ \text{such that there is an active path from}~(x,0)~\text{to}~~(y,t)\Big\},
\end{equation*}
\begin{equation*}
    \text{resp. } \mathcal{A}_t(\xi)= \Big\{y\in \Z^d,~ \exists x\in \mathcal{A}(\xi)~ \text{such that there is an active path from}~(x,0)~\text{to}~~(y,t)\Big\},
\end{equation*}
\begin{equation*}
    \text{resp. } \mathcal{A}_t(\zeta)= \Big\{y\in \Z^d,~ \exists x\in \mathcal{A}(\zeta)~ \text{such that there is an active path from}~(x,0)~\text{to}~~(y,t)\Big\}.
\end{equation*}
We refer to \cite[Section 2]{SaintFlour} for a proof that the graphical construction in the case of the contact process is well defined, and it adapts here for the graphical construction of the $IS(\lambda,p)$ process.

\subsection{Couplings}\label{COUPLINGS}
The graphical representations of processes allow to build the so called \textit{basic couplings}. They essentially consist in using some common Poisson Processes in their graphical representation. We use these coupling to prove Theorems \ref{Compcontact} and \ref{CouplingISwithSpont}.\medskip

\begin{proof}[Proof of Theorem \ref{Compcontact}]
    Take $\lambda>0$ and $p\in [0,1]$. Consider two configurations $(\e,\zeta)\in \Omega\times \{0,1\}^{\Z^d}$ such that $\e\leq \zeta$. Sample independent families of Poisson point processes $(N^1_{x,y})_{(x,y)\in E(\Z^{d})}$, $(N^2_{x,y})_{(x,y)\in E(\Z^{d})}$, $(U_x)_{x\in \Z^d}$, $(V_x)_{x\in \Z^d}$ with respective parameters $\lambda p$, $\lambda(1-p)$ and $1$. Deduce the evolution  of an $IS(\lambda,p)$ process $(\e_t)_{t\geq 0}$ starting from $\e_0$, and that of a contact process $(\zeta_t)_{t\geq 0}$ with parameter $\lambda p$ starting from $\zeta_0$, by using their graphical representation and using the same Poisson point processes for that. 
    
    When a site $x\in \Z^d$ flips from $0$ to $1$ in $(\e_t)_{t\geq 0}$, so does this happen for $(\zeta_t)_{t\geq 0}$, if $x$ was in state $0$ in $(\zeta_t)_{t\geq 0}$. Indeed, for such a flip to happen at $s>0$, there must be a $y\sim x$, such that an arrow $\overset{1}{\longrightarrow}$ is produced by $N^1_{y,x}$ for the graphical construction of $IS(\lambda,p)$, and such that $\e_{s^{-}}(y)=1$. This implies that $\zeta_{s^{-}}(y)=1$, so the arrow $\overset{1}{\longrightarrow}$ is also produced by $N^1_{y,x}$ for the graphical construction of the contact process with parameter $\lambda p$. Thus, $x$ flips from $0$ to $1$ in $(\zeta_t)_{t\geq 0}$. Furthermore, as we use the same Poisson point processes $(V_x)_{x\in \Z^d}$ for the flips from $1$ to $0$, such flips in $(\zeta_t)_{t\geq 0}$ happen simultaneously in $(\e_t)_{t\geq 0}$ at $x$, provided that $\zeta_t(x)=\e_t(x)=1$. If $\e_t(x)=0$ and $\zeta_t(x)=1$, the flip from $1$ to $0$ in $x$ happens only in $\zeta$. Therefore, flips from $0$ to $1$ in $(\e_t)_{t\geq 0}$ and flips from $1$ to $0$ in $\zeta_t$ can never disrupt the order, so the basic coupling is order preserving.

    \medskip
    
    In terms of transition rates, the basic coupling between $IS(\lambda,p)$ and the contact process with parameter $\lambda$ goes as follows. At site $x\in \Z^d$, for a current configuration $(\e,\zeta)$ such that $\e\leq \zeta$:
    \begin{equation}\label{CouplingIScontact}
    \begin{split}
        &(1,1) \rightarrow (0,0): 1,~~~~~~~~~~~~~~~~~~~~~(0,1) \rightarrow \left\{
    \begin{array}{ll}
        (0,0):~ 1\\
        (1,1):~  \lambda p n_1(x,\e)\\ 
        (-1,1):~   \lambda(1-p)n_1(x,\e)
        \end{array}
            \right.,\\\\
        &(0,0) \rightarrow  \left\{
    \begin{array}{ll}
        (1,1):~   \lambda p n_1(x,\e)\\ 
        (0,1):~   \lambda p  n_1(x,\zeta) -   \lambda p n_1(x,\e)\\ 
        (-1,0):~   \lambda(1-p)n_1(x,\e)
        \end{array}
        \right.,~~~
        (-1,1) \rightarrow  \left\{
    \begin{array}{ll}
        (0,1):~ 1\\ 
        (-1,0):~   1 
        \end{array}
\right.,\\\\
&(-1,0) \rightarrow  \left\{
    \begin{array}{ll}
        (-1,1):~    \lambda p  n_1(x,\zeta)\\ 
        (0,0):~   1 
        \end{array}
\right.,
    \end{split}
\end{equation}
where we recall that $n_1(x,\e)$, resp. $n_1(x,\zeta)$, is the number of neighbours of $x$ in state $1$ in $\e$, resp. $\zeta$. Since $\e\leq \zeta$, the rates are positive and the transition rates are well defined. 
\end{proof}

\begin{proof}[Proof of Theorem \ref{CouplingISwithSpont}]
    Take $\lambda>0$ and $p\in [0,1]$. Consider two configurations $(\e_0,\xi_0)\in \Omega^2$ such that $\xi_0\leq \eta_0$. As in the previous proof, sample independent families of Poisson point processes $(N^1_{x,y})_{(x,y)\in E(\Z^{d})}$, $(N^2_{x,y})_{(x,y)\in E(\Z^{2})}$, $(U_x)_{x\in \Z^d}$, $(V_x)_{x\in \Z^d}$ with respective parameters $\lambda p$, $\lambda(1-p)$ and $1$. Deduce the evolution  of an $IS(\lambda,p)$ process $(\e_t)_{t\geq 0}$ starting from $\e_0$, and that of a $Spont(\lambda,p)$ process $(\xi_t)_{t\geq 0}$ starting from $\xi_0$, by using their graphical representation and using the same Poisson point processes for that. 

    When a site $x\in \Z^d$ flips from $0$ to $1$ in $(\xi_t)_{t\geq 0}$, so does this happen for $(\e_t)_{t\geq 0}$, if $x$ was in state $0$, as this happens under the effect of the same arrows $\overset{1}{\longrightarrow}$ produced by $N^1_{y,x}$. Furthermore, when a site $x\in \Z^d$ flips from $0$ to $-1$ in $(\zeta_t)_{t\geq 0}$, so does this happen for $(\xi_t)_{t\geq 0}$, if $x$ was in state $0$. Indeed, for such a flip to happen at $s>0$, there must be a $y\sim x$ such that an arrow $\overset{-1}{\longrightarrow}$ is produced by $N^2_{y,x}$ for the graphical construction of $IS(\lambda,p)$, but this arrow is activated for $(\xi_t)_{t\geq 0}$ at $s$, whatever the state of $y$. Finally, as we use the same processes $(U_x)_{x\in \Z^d}$, resp. $(V_x)_{x\in \Z^d}$, for the flippings of $1$ to $0$, resp. $-1$ to $0$ in $\e$ and $\xi$, such flips happen simultaneously. Therefore, flips from $0$ to $1$ in $(\xi_t)_{t\geq 0}$, from $1$ to $0$ in $(\zeta_t)_{t\geq 0}$, from $0$ to $-1$ in $(\e_t)_{t\geq 0}$ and from $-1$ to $0$ in $(\xi_t)_{t\geq 0}$, can never disrupt the order, so the basic coupling is order preserving.\medskip

    In terms of transition rates, the basic coupling between $IS(\lambda,p)$ and a $Spont(\lambda,p)$ process goes as follows. At site $x\in \Z^d$, for a current configuration $(\e,\zeta)$ such that $\e\leq \zeta$:
        \begin{equation}\label{CouplingSpontandIS}
    \begin{split}
        &(0,0) \rightarrow \left\{
    \begin{array}{ll}
        (1,1):~ \lambda p n_1(x,\xi)\\ 
        (0,1):~ \lambda p\big[n_1(x,\e)-n_1(x,\xi) \big]\\
        (-1,-1):~ \lambda(1-p)n_1(x,\e)\\
        (-1,0):~\lambda(1-p)\big[2d-n_1(x,\e) \big]
        \end{array}
        \right.,~~~~(0,1) \rightarrow \left\{
    \begin{array}{ll}
        (0,0):~ 1\\
        (-1,1):~ 2d\lambda(1-p)\\
        (1,1):~ \lambda p n_1(x,\xi)
        \end{array}
        \right.,\\\\
        &(-1,1) \rightarrow \left\{
    \begin{array}{ll}
        (-1,0):~ 1\\
        (0,1):~ 1
        \end{array}
        \right.~~,~~~~~~~~~~~~~~~~~~~~~~~~~~~~~(-1,-1) \rightarrow (0,0): 1,\\\\
        &(-1,0) \rightarrow \left\{
    \begin{array}{ll}
        (0,0):~ 1\\
        (-1,1):~ \lambda p n_1(x,\e)\\
        (-1,-1):~ \lambda(1-p)n_1(x,\e)
        \end{array}
        \right.,~\mathrm{and}~~~~~~~~~~~~ (1,1) \rightarrow (0,0):1.
    \end{split}
\end{equation}
Since $\xi_0\leq \e_0$, the rates are positive and the dynamics is well defined. One can check that $(\xi_t)_{t\geq 0}$ is a $Spont(\lambda,p)$ process on $\Z^d$ and $(\eta_t)_{t\geq 0}$ is an $IS(\lambda,p)$ process with parameter $\lambda p$ and that almost surely, for any $t\geq 0$, $\xi_t\leq \eta_t$.

\end{proof}

\subsection{Monotonicity for $Spont(\lambda,p)$, lack of monotonicity for $IS(\lambda,p)$}\label{loa}
As stated in Proposition \ref{MonotonicityofSpont}, contrary to $IS$ processes (see Proposition \ref{Nomonot} and its proof in this subsection), we have monotonicity in $p$ at fixed $\lambda$ for $Spont$ processes.
\begin{proof}[Proof of Proposition \ref{MonotonicityofSpont}] Again, the basic coupling provides an order preserving coupling. For that, consider Poisson point processes indexed by oriented edges and sites $N^1_{x,y}$, $\Tilde{N}^1_{x,y}$, $N^2_{x,y}$, $\Tilde{N}^2_{x,y}$, $U_x$, $V_x$, with respective parameters $\lambda p_1$, $\lambda (p_2-p_1)$,  $2d\lambda(1-p_2)$, $2d\lambda(p_2-p_1)$, $1$ and $1$. Then, build the graphical representation for $Spont(\lambda,p_1)$ by using $N^1_{x,y}$, $N^2_{x,y}+\Tilde{N}^2_{x,y}$, $U_x$ and $V_x$, and the graphical representation for $Spont(\lambda,p_2)$ by using $N^1_{x,y} + \Tilde{N}^2_{x,y}$, $N^2_{x,y}$, $U_x$ and $V_x$. As in the proofs in Section \ref{COUPLINGS}, one can check that this coupled graphical representation is order preserving.\medskip

The rates of this coupling are given as follows. At site $x\in \Z^d$ for a current couple of configuration $(\xi, \Tilde{\xi})$ such that $\xi \leq \Tilde \xi$:
\begin{equation}\label{CouplingSpontSpont}
    \begin{split}
        &(0,0) \rightarrow \left\{
    \begin{array}{ll}
        (1,1):~ \lambda p_1 n_1(x,\xi)\\ 
        (0,1):~ \lambda p_2n_1(x,\Tilde{\xi}) - \lambda p_1n_1(x,\xi) \\
        (-1,-1):~ 2d\lambda(1-p_2)\\
        (-1,0):~2d\lambda(p_2-p_1)
        \end{array}
        \right.,~~~~(0,1) \rightarrow \left\{
    \begin{array}{ll}
        (0,0):~ 1\\
        (-1,1):~ 2d\lambda(1-p_1)\\
        (1,1):~ \lambda p_1 n_1(x,\xi)
        \end{array}
        \right.,\\\\
        &(-1,1) \rightarrow \left\{
    \begin{array}{ll}
        (-1,0):~ 1\\
        (0,1):~ 1.
        \end{array}
        \right.~~,~~~~~~~~~~~~~~~~~~~~~~~~~~~~~(-1,-1) \rightarrow (0,0): 1\\\\
        &(-1,0) \rightarrow \left\{
    \begin{array}{ll}
        (0,0):~ 1\\
        (-1,1):~ \lambda p_2 n_1(x,\Tilde{\xi})\\
        (-1,-1):~ 2d\lambda(1-p_1)
        \end{array}
        \right.,~\mathrm{and}~~~~~~~~~~~~ (1,1) \rightarrow (0,0):1.
    \end{split}
\end{equation}

\end{proof}
Now let us prove Proposition \ref{Nomonot}, which claims that for any ordering of $\{-1,0,1\}$, there is no monotonicity in $p$ for an $IS$ process at fixed $\lambda$. First note that the basic coupling, built as in the proof of Proposition \ref{MonotonicityofSpont}, does not provide an order preserving coupling, whatever the order on $\{-1,0,1\}$. In fact, hereafter we only consider orders where $1$ is the maximal element of the set, as we are interested in the survival of $1$'s.
\begin{itemize}
    \item For the order $-1<0<1$: consider $\e^1,\e^2\in \Omega^2$ with $\e^1\leq \e^2$, such that there is $x\sim y\in \Z^d$ with $\e^1(x)=\e^1(y)=0$, $\e^2(x)=1$ and $\e^2(y)=0$. If an arrow $\overset{-1}{\longrightarrow}$ is produced by $N^2_{x,y}$, a birth of a $-1$ happens at $y$ for $\e^2$ which breaks the ordering of $\e^1$ and $\e^2$.
    \item For the order $-1<0<1$ and the partial order $0,-1<1$ : consider $\e^1,\e^2\in \Omega^2$ with $\e^1\leq \e^2$, such that there is $x\sim y\in \Z^d$ with $\e^1(x)=1$, $\e^1(y)=0$, $\e^2(x)=1$ and $\e^2(y)=-1$. If an arrow $\overset{1}{\longrightarrow}$ is produced by $N^1_{x,y}$, a birth of a $1$ happens at $y$ for $\e^1$ which breaks the ordering of $\e^1$ and $\e^2$.
\end{itemize}
This is not enough to conclude with the absence of monotonicity in $p$ as other couplings could be order preserving. It turns out that in \cite{Borrello}, a characterization of the monotinicity of a process is given in terms of conditions on its transition rates, see \cite[Theorem 2.4]{Borrello} that we recall in Appendix \ref{ApppendixBorrello}. We use this very convenient characterization here.

\begin{proof}[Proof of Proposition \ref{Nomonot}]
    Again, we discuss according to the ordering.
    \begin{itemize}
        \item For the order $-1<0<1$:
        Using the notation in \cite{Borrello} (see \ref{ApppendixBorrello} for more details), the birth and death rates are given by:
        \begin{equation}\label{BR1}
            R_{1,0}^{0,1}=\lambda p,~ ~ R_{0,1}^{-1,0} = \lambda (1-p),~~P_1^{-1}=P_1^1=1.
        \end{equation}
    Using the same notation as in the statement of Theorem 2.4 in \cite{Borrello}, taking $(\alpha,\beta)=(0,0)$, $(0,1)=(\gamma,\delta)$ and $h_1=0$, we have
$$ \sum_{k\in X, k>\gamma -\alpha} R_{\gamma,\delta}^{-k,0} = \lambda(1-p) > \sum_{k\in X, k>j_{1}} R_{\alpha,\beta}^{-k,0}=0,$$
so inequality (2.14) in the characterization of monotonicity in Theorem 2.4 of \cite{Borrello} is not satisfied.

        \item For the order $0<-1<1$: the birth and death rates are given by:
        \begin{equation}\label{BR2}
            R_{1,0}^{0,2}=\lambda p,~ ~ R_{1,0}^{0,1} = \lambda (1-p),~~P_1^{-2}=P_{-1}^{-1}=1.
        \end{equation}
        Now, taking $(\alpha,\beta)=(1,0)$, $(\gamma,\delta) = (1,-1)$ and $h_1=0$, we have
        $$ \sum_{k\in X, k>\delta-\beta} R_{\alpha,\beta}^{0,k} = \lambda p >\sum_{k\in X, k>0} R_{\gamma,\delta}^{0,k}=0,$$
        so inequality (2.13) in the characterization of monotonicity in Theorem 2.4 of \cite{Borrello} is not satisfied.

        \item For the partial order $0,-1<1$ one can take the birth or death rates to be as in \eqref{BR1} or \eqref{BR2}. In either cases, inequalities (2.13) or (2.14) in \cite{Borrello} are not satisfied and one does not have monotonicity.
         \end{itemize}
\end{proof}
\begin{rem}
    Using \cite[Theorem 2.4]{Borrello}, one can also show that there is no monotonicity in $\lambda$, at fixed $p$, for an $IS$ process as well as a $Spont$ process on $\Z^d$.
\end{rem}

\section{Phase transition for the $Spont$ process}\label{PMR}
\subsection{Proof of Theorem \ref{Phaseforspont}}
From the monotonicity of $Spont$, stated in Proposition \ref{MonotonicityofSpont}, the following holds
\begin{cor}\label{increaseofproba}
    Suppose that $\lambda>0$ is fixed and consider $\xi\in \Omega$. The mapping
    $$p \mapsto \Tilde{\mathbb{P}}_{\xi}^{\lambda,p}\Big(\forall t\geq 0,~ A(\xi_t)\neq \emptyset \Big) $$
    is a non decreasing function, where we recall that $A(\xi_t)$, defined in \eqref{setoffertile}, is the set of sites in state $1$ in $\xi_t$.
\end{cor}
\begin{proof}
    Let $p_1<p_2$ and consider an order preserving coupling of $(\xi_t,\Tilde{\xi}_t)_{t\geq 0}$ on $\Omega^2$ initially in $(\xi,\xi)$ such that $(\xi_t)_{t\geq 0}$ is a $Spont(\lambda,p_1)$ process and $(\Tilde{\xi}_t )_{t\geq 0}$ a $Spont(\lambda,p_2)$ process. Then for any $t\geq 0$, $A(\xi_t)\subset A(\Tilde{\xi}_t)$, hence the result.
\end{proof}

\begin{prop}\label{highenoughp}
    Fix $\lambda>\lambda_c(d)$. For $p<1$ large enough, the process $Spont(\lambda,p)$ survives.
\end{prop}
The proof of Proposition \ref{highenoughp} is the object of Section \ref{Proofhighenough}.
From Corollary \ref{increaseofproba} and Proposition \ref{highenoughp}, we deduce the proof of Theorem \ref{Phaseforspont}.

\begin{proof}[Proof of Theorem \ref{Phaseforspont}]
    Consider $\lambda>\lambda_c(\Z^d)$ and introduce 
    \begin{equation*}
        p_c^{spont}(\lambda):=\inf\{~ p\in [0,1),~ ~Spont(\lambda,p)~ \text{survives}~ \}.
    \end{equation*}
    By Proposition \ref{highenoughp}, $p_c^{spont}(\lambda)<1$. By Corollary \ref{increaseofproba}, for any $p>p_c^{spont}(\lambda)$ a $Spont(\lambda,p)$ process survives and, for any $p<p_c^{spont}(\lambda)$ a $Spont(\lambda,p)$ becomes extinct. Furthermore, building an order preserving coupling between $Spont(\lambda,p)$ and a contact process with birth parameter $\lambda p$ in the same spirit as the coupling \eqref{CouplingIScontact} between an $IS(\lambda,p)$ and a contact process, we get that $p_c^{spont}(\lambda)\geq \lambda_c(d)/\lambda$.
\end{proof}

\subsection{Proof of Proposition \ref{highenoughp}}\label{Proofhighenough}



In order to prove Proposition \ref{highenoughp}, that is, that for $\lambda>\lambda_c(d)$ and for $p \in [\lambda_c(d)/\lambda,1)$, large enough, the process $Spont(\lambda,p)$ survives, we use a comparison with oriented percolation Theorem. For that, we rely on the graphical construction of our processes (see Section \ref{graphrepr}). The idea underlying the Comparison Theorem (see Theorem \ref{ComparisonT}), is to show that for $p$ large enough, the process dominates an oriented percolation configuration containing, almost surely, an infinite component. 

In what follows, we recall the definition of oriented percolation and state the Comparison Theorem. We also recall some results on the contact process. Then, we apply the Comparison Theorem to the $Spont$ process.

\subsubsection{Comparison Theorem }\label{SectioncompT}
Let us recall the definition of oriented site percolation in two dimensions. We refer to \cite{SaintFlour} and references therein for the proofs of the results on oriented site pecolation stated below.

The underlying graph for oriented site percolation with parameter $p\in [0,1]$ is the graph with vertices the bi-dimensional even lattice
\begin{equation}\label{Evengrid}
    \mathcal{L} = \Big\{(m,n) \in \Z^2,~ m+n~\text{is even },~ n\geq 0 \Big\}
\end{equation}
and with edges the oriented bonds
$$(m,n) \rightarrow (m+1,n+1),~\text{and } (m,n) \rightarrow (m-1,n+1).$$
An oriented site percolation graph is obtained by keeping each site $(m,n)\in \mathcal{L}$ with probability $p$ and discarding it with probability $1-p$ (there might be some dependencies in the samplings of sites but we will discuss this further). We say that the site is open if it has been kept after sampling and closed otherwise. \medskip

We say that there is an oriented open path from $(x,n)$ to $(y,m)$ and denote this by $(x,n)\rightarrow (y,m)$ if there exists a sequence of points $x=x_1,...,x_k=m$ such that $(x_i,n+i) \in \mathcal{L}$, $|x_i-x_{i+1}| = 1$ for $1\leq i \leq k-1$ and the sites $(x_i,n+i)$ are all open.

Given an initial set of open sites $\mathcal{A}_0 \subset 2\Z$ we denote by $\mathcal{A}_n$ the following set of sites:
$$ \mathcal{A}_n = \{y,~ (x,0) \rightarrow (y,n)~ \text{for some } x\in \mathcal{A}_0\},$$
that is, the set of attainable sites at time $n$, starting from those in $\mathcal{A}_0$.\medskip

Let $\mathcal{A}_n^0$ be the set of reachable sites at time $n$ when $\mathcal{A}_0=\{0\}$ and define
$\mathcal{C}_0 = \underset{n\geq 0}{\cup} \mathcal{A}_n $, that is, the set of points reached by the origin through a connected open oriented path. We say that percolation occurs when $|\mathcal{C}_0| = \infty$.
\begin{thm}{Percolation for independent samplings.}\label{Percoindep} 

 Suppose that the samplings of sites are performed independently from one another. Then, for $p\in [0,1)$ large enough,
$$\mathbb{P}[~ |\mathcal{C}_0| = \infty] >0. $$
\end{thm}
The proof of this can be obtained thanks to a Peierls argument (or dual contour argument) and we refer to \cite{Grimmett}, or \cite{SaintFlour}. In \cite{SaintFlour}, it is detailed how Theorem \ref{Percoindep} can be extended to the case where samplings are not necessarily independent but with finite range dependencies. 
\begin{Def}
    Fix $M>0$ an integer. We say that the samplings of sites are $M$-dependent with intensity at least $1-\gamma$ (with $\gamma \in [0,1]$) if, whenever $(m_i,n_i)_{1\leq i \leq k}$ is a finite sequence such that $\|(m_i,n_i) - (m_j,n_j)\|_{\infty}>M$ for $i\neq j$, then
    \begin{equation}\label{M}
        \mathbb{P}\big[\underset{1\leq i \leq k}{\cup}(n_i,m_i) \text{ is open } \big]\geq 1-\gamma^k.
    \end{equation}
\end{Def}
\begin{thm}{Percolation for $M$-dependent samplings.}\ \label{Criticvalueperco}

Consider an $M$-dependent percolation process with intensity at least $1-\gamma$. If $\gamma \leq 6^{-4(2M+1)}$, then
$$ \mathbb{P}[~ |\mathcal{C}_0| = \infty] >0.  $$

\end{thm}
Again, we refer to \cite{SaintFlour} for a detailed proof of Theorem \ref{Criticvalueperco}.\medskip

The Comparison Theorem gives general conditions which guarantee that an interacting particle system dominates an oriented site percolation. This domination relation allows to infer survival of the process if there is an infinite path starting from the origin in the oriented percolation. We refer to the seminal paper \cite{Bramson} where this technique is used for spin systems.\medskip

Consider $(\xi_t)_{t\geq 0}$ a translation invariant and finite range process with state space $X^{\Z}$, which can be constructed from a graphical representation. Fix some positive integers $N,T,k$ and $j$. For $(m,n)\in \mathcal{L}$, define the space-time regions
\begin{equation}\label{kandj}
   R_{m,n} = (2mN e_1,nT) + [-kN,kN]^{d}\times [0,jT], 
\end{equation}
where $e_1$ denotes the first vector in the canonical basis of $\Z^d$. Let $M= \max(k,j)$ so that the regions $R_{m,n}$ and $R_{m',n'}$ are disjoint as soon as $\|(m,n)- (m',n')\|_{\infty}>M$. Let $H$ be the set of configurations satisfying a certain property which only depends on the state of sites in $[-N,N]^{d}$. For $\xi\in H$, we say that $(m,n)\in \mathcal{L}$ is wet if
\begin{equation}\label{wetness}
    \tau_{-2mN e_1}\xi_{nT}\in H ~~\Longrightarrow~~ \tau_{-2(m-1)N e_1}\xi_{(n+1)T}\in H~\text{ and }~\tau_{-2(m+1)N e_1 }\xi_{(n+1)T}\in H,
\end{equation}
where $\tau_x$ stands for the translation by $x$ and where $(\xi_t)_{t\geq 0}$, is the process starting from $\xi$. Introduce the set $X_n^{\xi}$, defined by
\begin{equation}
    X_n^{\xi} = \{m,~ (m,n)\in \mathcal{L}~\text{is wet}\}.
\end{equation}

The idea underlying the comparison theorem is to overlap the graphical representation of the process with the graph $\mathcal{L}$, and to prove that the set of wet sites in $\mathcal{L}$ stochastically dominates the set of open sites of an $M$-dependent percolation configuration on $\mathcal{L}$. For that, note that the boxes $R_{m,n}$ and $R_{m',n'}$ are disjoint when $\|(m,n)- (m',n')\|_{\infty}>M$, and that the probability for a site in $(m,n)$ to be wet can be made large enough, by using a so called good event (see its definition in Theorem \ref{ComparisonT}). This allows us to conclude that there are infinitely often wet sites, and the population survives infinitely often.\medskip
 
\begin{thm}{\cite[Section 4]{SaintFlour}}\ \label{ComparisonT} Suppose that the following conditions are satisfied:
\begin{itemize}
    \item [(i)] For any $\xi\in H$, there is a good event $G^{\xi}$ which is measurable with respect to the graphical representation in $R_{0,0} = [-kN,kN]^d\times [0,jT]$, such that if $\xi_0=\xi$ then, on $G^{\xi}$, $\xi_T\in \tau_{-2Ne_1} H\cap \tau_{2Ne_1}H$.
    \item [(ii)] There is a $\gamma\in [0,1]$ such that for any $\xi\in H$, $\mathbb{P}(G^{\xi})\geq 1-\gamma$.
\end{itemize}
Then, for any $\xi\in H$,
$X_n^{\xi}$ dominates a two dimensional $M$-dependent oriented site percolation (recall that $M=\max(k,j)$) with initial configuration $\mathcal{A}_0=A(\xi)$ and density at least $1-\gamma$, that is,$$\forall n\geq 0, ~ \mathcal{A}_n \subset X_n~~a.s. $$
\end{thm}
\begin{cor}
    From Theorems \ref{ComparisonT} and \ref{Criticvalueperco}, one can deduce that if $\gamma$ can be made arbitrarily small, the process $(\xi_t)_{t\geq 0}$ starting from a configuration $\xi\in H$ survives with strictly positive probability. In particular, if there is a strictly positive probability of connecting the configuration $\{0\}$ to a configuration $\xi\in H$ in a finite space time box, one deduces that the process $(\xi_t)_{t\geq 0}$ starting from  $\{0\}$ survives with strictly positive probability.
\end{cor}

To apply the comparison Theorem, as will be done in Section \ref{PROOF}, one needs to properly choose $H$ and fix $k,j,N,T$. Then, one is left to check that the comparison assumption holds.

\subsubsection{Preliminary results on the contact process}

Given a subset $\mathcal{A}$ of $\Z^d$, let $(\zeta_t^{A})_{t\geq 0}$ be a contact process starting with $1$'s in each site of $\mathcal{A}$ and $0$'s everywhere else. For a configuration $\zeta\in \{0,1\}^{\Z^d}$, denote by $|\zeta|\in [0,\infty]$ the number of ones (possibly infinite) in $\zeta$. Also, denote by 
\begin{equation*}
    \tau^{\mathcal{A}}(x) := \inf\big\{t\geq 0,~ \zeta_t^{\mathcal{A}}(x) = 1\big\},~~\text{and}~~
   \mathcal{H}_t^{\mathcal{A}}:= \big\{x\in \Z^d,~ \tau^{A}(x)\leq t\big\},
\end{equation*}
the set of sites which have been occupied before time $t$. The following result, for which a proof can be found in \cite{CPseveraldimensions}, tells us that the set $\mathcal{H}_t^{\mathcal{A}}$ is at most linearly growing.
\begin{prop} \label{ALL}
    Fix $\lambda>\lambda_c(d)$. There exists $\alpha_1>0$, such that for any finite set $\mathcal{A}\subset \Z^d$, there are $C_{\mathcal{A}},\ell_{\mathcal{A}}>0$ such that for any contact process with parameter $\lambda$, 
    \begin{equation}
        \forall t\geq 0,~ \forall x\notin \mathcal{A} + [-\alpha_1 t, \alpha_1 t]^d,~ \mathbf{P}_{\zeta^{\mathcal{A}}}^{\lambda}\Big(x\in \mathcal{H}_t^{\mathcal{A}} \Big) \leq C_{\mathcal{A}}e^{-\ell_{\mathcal{A}}t}.
    \end{equation}
\end{prop}
The following result tells us that unless a contact process $(\zeta_t^{A})_{t\geq 0}$ is extinct at time $t$, it is coupled to $(\zeta_t^{\Z^d})_{t\geq 0}$ inside a linearly growing set. This result, stated in one dimension in \cite{linear}, is generalized in \cite{Garet_Marchand} (see equation (44)) in any dimension thanks to the estimates in \cite{BG}. 

\begin{prop}{\cite{linear, Garet_Marchand}}\label{Schinaz}
Consider a contact process $(\zeta_t)_{t\geq 0}$ with parameter $\lambda>\lambda_c(d)$. There exists $\alpha_2>0$ such that for 
any finite set $\mathcal{A}\subset \Z^d$, there are $C'_\mathcal{A},\ell'_\mathcal{A}>0$ 
 such that at time $t>0$,
 \begin{equation*}
     \forall x\in \mathcal{A} + [-\alpha_2 t, \alpha_2 t]^d,~  \mathbf{P}_{\zeta^{\mathcal{A}}}^{\lambda}\Big(|\zeta_t^{\mathcal{A}}| \neq 0 \cap \zeta_t^{\mathcal{A}}(x) \neq \zeta_t^{\Z^d}(x)  \Big) \leq C'_\mathcal{A} e^{-\ell'_\mathcal{A} t}.
 \end{equation*}
 \end{prop}
\begin{prop}{\cite[Proposition 2.1, Chapter 2]{Liggett2} } \label{-n,n}
    Fix $\lambda>\lambda_c(d)$ and consider a contact process with parameter $\lambda$. Then,
\[
        \lim_{n \rightarrow \infty} \mathbf{P}_{\zeta^{[-n,n]^d}}^\lambda \left(\forall t \geq 0,~ |\zeta_t^{[-n,n]^d}| \neq 0 \right) = 1.
    \]
\end{prop}
We recall the proof of the above result as it requires a relation that will be used later on.
\begin{proof}
   Denote by $\overline{\nu}$ the upper invariant measure of the contact process on 
 $\Z^d$ with parameter $\lambda$. The graphical representation of the contact process provides the following self-dual relation, that can be found for instance in (1.7), \cite[Part I]{Liggett2}: for a configuration $\zeta$, for any $t\geq 0$,
\begin{equation}\label{Dualrelationcontact1}
\begin{split}
        \textbf{P}_{\zeta}^{\lambda}\Big(A(\zeta_t)\neq \emptyset \Big) &= \textbf{P}_{\zeta^{\Z^d}}^{\lambda}\Big(A(\zeta)\cap A(\zeta^{\Z^d}_t)\neq \emptyset  \Big).
\end{split}
\end{equation}
Taking $t$ to infinity on both sides of the above equality and since $A(\zeta_t)\neq \emptyset $ is a decreasing event, we get that
\begin{equation}\label{Dualrelationcontact0}
\begin{split}
        \textbf{P}_{\zeta}^{\lambda}\Big(A(\zeta_t)\neq \emptyset,~  \forall t\geq 0\Big) &= \underset{t\rightarrow \infty}{\lim}\textbf{P}_{\zeta^{\Z^d}}^{\lambda}\Big(A(\zeta)\cap A(\zeta^{\Z^d}_t)\neq \emptyset\Big).
\end{split}
\end{equation}
Using the convergence in law of $\zeta^{\Z^d}_t$ to $\overline{\nu}$, and the notation identifying a configuration in $\{0,1\}^{\Z^d}$, to its set $S$ of ones, we are left with
\begin{equation}\label{Dualrelationcontact}
\begin{split}
        \textbf{P}_{\zeta}^{\lambda}\Big(A(\zeta_t)\neq \emptyset,~  \forall t\geq 0\Big) = \overline{\nu}\Big(S,~\text{s.t}~ S\cap A(\zeta)\neq \emptyset \Big).
\end{split}
\end{equation}
Taking $A(\zeta)=[-n,n]^d$, and then $n \rightarrow \infty$ yields the result, as the upper invariant measure is non trivial. 
\end{proof}

\subsubsection{Proof of Proposition \ref{highenoughp} }\label{PROOF}
We are now in position to apply the comparison theorem \ref{ComparisonT}.
Fix $\lambda>\lambda_c(d)$ and $p\in [\lambda_c(d)/\lambda,1)$. Given an integer $N$, define the space time box
\begin{equation}\label{spacetimeboxes}
    R=(-8N,8N)^d\times [0,T],
\end{equation}
with
\begin{equation}\label{Timedivision}
    T = \frac{3N}{\alpha'}=\frac{N}{2\alpha_1} + \big(\frac{3N}{\alpha'}-\frac{N}{2\alpha_1}\big)=: T_1 + T_2.
\end{equation}
Above, $\alpha_1$, is given by Proposition \ref{ALL}, where we replace $\lambda$ by $\lambda p$, and $\alpha'=\min(6\alpha_1, \alpha_2)$, where $\alpha_2$ is given by Proposition \ref{Schinaz}, also replacing $\lambda$ by $\lambda p$. Consider the space box
\begin{equation}\label{Spaceintervals}
    I=[-2N,2N]^d,
\end{equation}
and keep in mind that $N$ and $T$ will be taken large.  With the notation introduced in Section \ref{SectioncompT}, this corresponds to having $k=8$ and $j=\lfloor\frac{3}{\alpha'}\rfloor +1$. Denote by $M=\max(k,j)$.

By Proposition \ref{-n,n}, we can choose $K\leq N$ (with $N$ large enough) such that if $(\zeta_t^{[-K,K]^d})_{t\geq 0}$ is a supercritical contact process with parameter $\lambda p$ starting from $[-K,K]^d$ filled with $1$'s and $0$'s everywhere else,
\begin{equation}\label{Keps}
    \mathbf{P}^{\lambda p}_{\zeta^{[-K,K]^d}}\Big(\forall t\geq 0,~ \big|\zeta_t^{[-K,K]^d}\big|\neq 0 \Big)> 1-\gamma/2,
\end{equation}
where $\gamma$ is a constant depending only on $M$ and given by Theorem \ref{Criticvalueperco}.\medskip

Given a subset $\mathcal{A}$ of $\Z^d$ and a configuration $\xi\in \Omega$, denote by $\xi^{|\mathcal{A}}$ the configuration defined by
\begin{equation}\label{defrestict}
    \xi^{|\mathcal{A}}(x)=\begin{cases}
        \xi(x),~ \text{if}~ x\in \mathcal{A}\\
        0~~\text{otherwise}.
    \end{cases}
\end{equation}
\begin{Def}
    Denote by $H$ the set of configurations $\xi\in \Omega$ such that there are no $-1$'s in $I$ and there is a translate $\mathcal{C}$ of $[-K,K]^d$ in $[-N,N]^d$, such that the probability that the contact process starting from $\xi^{|\mathcal{C}}$ survives, is greater than $1-\frac{\gamma}{2}$, i.e., 
    \begin{equation}
        \mathbf{P}^{\lambda p}_{\xi^{|\mathcal{C}}}\Big(|\zeta_t|\neq 0,~ \forall t\geq 0 \Big)>1-\frac{\gamma}{2}.
    \end{equation}
\end{Def}
By choice of $K$, $H$ is not empty.

\begin{lem}\label{Restriction}
    Given $(\xi_t)_{t\geq 0}$ a $Spont(\lambda,p)$ process, denote by $(\xi_t^R)_{t\geq 0}$ the restriction of $(\xi_t)_{t\geq 0}$ to the space time region $R$, that is, constructed from the graphical representation where only arrival times of the Poisson processes occurring within $R$ are taken into account. If
    \begin{equation*}
        \forall x\in R,~ \xi_0^R(x)\leq \xi_0(x),
    \end{equation*}
    then, a.s. for all $t>0$,
    \begin{equation*}
        \forall x\in R,~ \xi_t^R(x)\leq \xi_t(x).
    \end{equation*}
\end{lem}
\begin{proof}
    By construction, deaths within $R$ produce the same effect for $\xi_t^R$ and $\xi_t$. If a $-1$ clock rings, at time $t$ on a site $x\in R$ such that $\xi_{t^-}^R=1$, that means that we necessarily had $\xi_{t^-}(x)=1$, so both these sites are already occupied and no $-1$ appears on $x$. A birth on a site $x\in R$ from some site $y$ only occurs for $\xi_{t}^R$ if $y\in R$ but then it would also occur for $\xi_{t}$.
\end{proof}
Now, introduce the good event as follows. For $\xi\in H$,
\begin{equation*}
    G^{\xi} = \Big(\text{if }\xi_0=\xi,~~\text{then}~~\xi_T^R\in \tau_{-2Ne_1}H \cap \tau_{2Ne_1}H \Big).
\end{equation*}
By definition, $(\xi_t^R)_{0\leq t\leq T}$ evolves only according to the space time events in $R$, so $G^{\xi}$ is measurable with respect to the graphical representation in $R$. By Lemma \ref{Restriction}, on the event $G^{\xi}$, we have that $\xi_T\in \tau_{-2Ne_1}H \cap \tau_{2Ne_1}H$. Therefore, the first assumption in the comparison theorem (see Theorem \ref{ComparisonT}) is satisfied. To check the second assumption, we show that for large enough $N$ and then large enough $p$,
\begin{equation}\label{lowerboundprobagoodevent}
    \mathbb{P}(G^{\xi}) \geq 1-\gamma.
\end{equation}
  \begin{figure}
        \begin{center}
        \centering
            \includegraphics[scale=0.32]{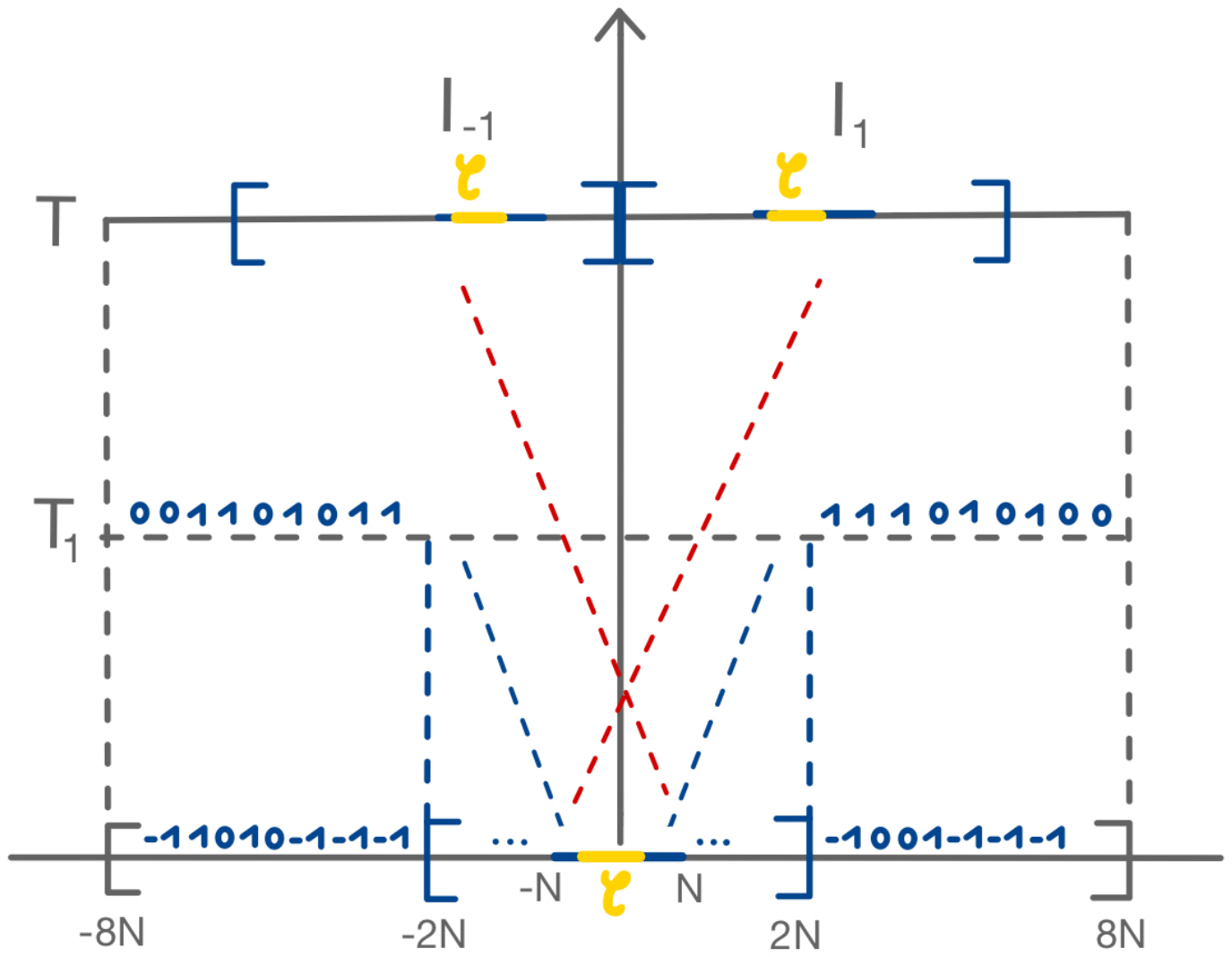}
        \end{center}
        \caption{At $t=0$ there is a translate $\mathcal{C}$ of $[-K,K]$ such that the contact process restricted to $\mathcal{C}$ survives with probability greater than $1-\gamma/2$. At time $T_1$, with high probability, all $-1$'s initially present have died and, up to time $T$, no new $-1$'s have appeared. From $t=0$ to $T_1$, the leftmost, resp. rightmost $1$, has not reached $-2N$, resp. $2N$ (blue dashed lines). From $t=0$ to $T$, w.h.p, the leftmost, resp. rightmost $1$ in a supercritical contact process with parameter $\lambda p$ starting from $\mathcal{C}$, has reached $I_1$, resp. $I_{-1}$ (red dashed line). At $T$, the supercritical contact process is not extinct with probability greater than $1-\gamma/2$ and, w.h.p, on $[-3N,3N]$ it is coupled to a contact process where all sites are initially filled with $1$'s. }
    \end{figure}
Define the following events, measurable with respect to the graphical representation in $R$:
\begin{equation*}
    E_1^{\xi}:=\Big(\text{No} -1~ \text{appears in } R~ \text{after time } 0 \Big),
\end{equation*}
\begin{equation*}
    E_2^{\xi} := \Big(\text{If type} -1~\text{individuals are present in } \xi \text{ in } (-8N,8N)^d \setminus I,~ \text{they all die by time } T_1 \Big).
\end{equation*}
Denoting by $A_{N,R}(\lambda,p)$ the first arrival  time of a Poisson process in $[-8N,8N]^d\times [0,T]$ with rate $2\lambda(1-p)$,
\begin{equation*}
      \mathbb{P}\big(E_1^{\xi}\big) = \mathbb{P}\big(A_{N,R}(\lambda,p)>T\big) = e^{-2\lambda(1-p)(16N+1)^dT}.
\end{equation*}
Moreover, as type $-1$'s individuals die at rate $1$,
\begin{equation*}
    \mathbb{P}\big(E_2^{\xi}|E_1^{\xi}\big) = \Big(1-\exp(-T_1)\Big)^{(12N)^d}.
\end{equation*}
Therefore,
\begin{equation}\label{E1&E2}
    \mathbb{P}\big(E_2^{\xi}\cap E_1^{\xi}\big) = e^{-2\lambda(1-p)(16N+1)^dT}\Big(1-\exp(-T_1)\Big)^{(12N)^d}.
\end{equation}
 Let us define two other intermediate events.
Denote by $(\zeta_t^{|\mathcal{C}})_{0\leq t \leq T}$ the contact process with birth rate $\lambda p> \lambda_c(d)$ starting from $\xi^{|\mathcal{C}}$, (recall the definition in \eqref{defrestict}) and evolving according to the graphical representation of $Spont(\lambda,p)$ on $\mathbb{Z}^d$, but ignoring the $-1$ crosses. Introduce,
\begin{equation*}
    E_3^{\xi}:=\Big(\text{by time } T_1, \text{ the } 1's \text{ in } \zeta_{T_1}^{|\mathcal{C}} \text{ have not reached the boundaries of } [-2N,2N]^d\Big),
\end{equation*}
\begin{equation*}
    E_4^{\xi}:=\Big(\forall x\in [-3N,3N]^d + A(\zeta_0^{|\mathcal{C}}),~ \zeta_T^{|\mathcal{C}}(x) =\zeta_T^{\Z^d}(x) \Big).
\end{equation*}
Denoting by $\partial$ the boundary points of a given box, writing $\mathcal{A}$ instead of $\mathcal{A}(\xi)$ and using Proposition \ref{ALL}, 
 \begin{equation}\label{E_3proba_1}
\begin{split}
\mathbb{P}\big(E_3^{\xi} \big) &\geq 1- \mathbb{P}\Big(\exists x \in \partial [-2N,2N]^d,~ x\in \mathcal{H}_{T_1}^{\mathcal{A}}\Big)\\
&\geq 1-C_{K}N^{d-1}e^{-\ell_{K}T_1},
\end{split}
\end{equation}
where $C_K$ and $\ell_K$ are constants depending only on $K$ and which might change from line to line in the sequel. Let us now lower bound $\mathbb{P}\big(E_4^{\xi}\big)$. By Proposition \ref{Schinaz}, 
\begin{equation}
    \begin{split}\label{E_4proba_1}
       \mathbb{P}\big(E_4^{\xi}\big)&= 1- \mathbb{P}\Big(\exists x\in [-3N,3N]^d + A(\zeta_0^{|\mathcal{C}}),~\zeta_T^{|\mathcal{C}}(x) \neq \zeta_T^{\Z^d}(x) \Big)\\
        &\geq 1- \mathbb{P}\Big(\exists x\in [-3N,3N]^d + A(\zeta_0^{|\mathcal{C}}),~\zeta_T^{|\mathcal{C}}(x) \neq \zeta_T^{\Z}(x)\big) \cap \big(|\zeta_T^{|\mathcal{C}}|\neq 0 \big) \Big)- \mathbb{P}\big(|\zeta_T^{|\mathcal{C}}|= 0 \big)\\
        &\geq 1-(7N)^dC_Ke^{-\ell_K T}-\gamma/2.
    \end{split}
\end{equation}
Collecting \eqref{E1&E2}, \eqref{E_3proba_1} and \eqref{E_4proba_1}, we are left with
\begin{equation}
\begin{split}\label{firstboundG}
        \mathbb{P}\big(G^{\xi}\big)&\geq  e^{-2\lambda(1-p)(16N+1)^dT}\Big(1-\exp(-T_1)\Big)^{(12N)^d} \big(1- C_{K}N^{d}e^{-\ell_{K} T_1}-\gamma/2 \big)\times\mathbb{P}\big(G^{\xi}\big| E^{\xi}\big).
\end{split}
\end{equation}
On the event $E^{\xi}:=E_1^{\xi}\cap E_2^{\xi}\cap E_3^{\xi}\cap E_4^{\xi}$, 
\begin{equation}\label{DifferencewithIS}
    \forall (x,t)\in R,~ \xi^R_t(x) \geq \zeta_t^{|\mathcal{C}}(x),~ \text{a.s.}
\end{equation} Therefore, 
\begin{equation}
    \begin{split}
        \mathbb{P}\big(G^{\xi}\big| E^{\xi}\big)&\geq \mathbb{P}\Big(\zeta_T^{|\mathcal{C}}\in \tau_{-2Ne_1}H\cap \tau_{2Ne_1}H\big|E^{\xi}\Big)\geq \mathbb{P}\Big(\zeta_T^{\Z^d}\in \tau_{-2Ne_1}H\cap \tau_{2Ne_1}H\big|E^{\xi}\Big).
    \end{split}
\end{equation}
Using that $(\zeta_t^{|\mathcal{C}})_{0\leq t \leq T}$ is independent of the events $E_1^{\xi}$ and $E_2^{\xi}$ (the contact process ignores the $-1$ birth and death events), we get
\begin{equation}\label{GxiExi}
    \begin{split}
        \mathbb{P}\big(G^{\xi}\big| E^{\xi}\big)
        &\geq \mathbb{P}\Big(\zeta_T^{\Z^d}\in \tau_{-2Ne_1}H\cap \tau_{2Ne_1}H\big|E_3^{\xi}\cap E_4^{\xi}\Big)\\
        &\geq \mathbb{P}\Big(\zeta_T^{\Z^d}\in \tau_{-2Ne_1}H\cap \tau_{2Ne_1}H\Big) - \Big(1-\mathbb{P}\big(E_3^{\xi}\cap E_4^{\xi}\big)\Big)\\
        &\geq \mathbf{P}_{\zeta^{\Z^d}}^{\lambda p}\Big(\tau_{2Ne_1}\zeta_T^{\Z^d}\in H\Big)^2 - \Big(1-\mathbb{P}\big(E_3^{\xi}\cap E_4^{\xi}\big)\Big),
    \end{split}
\end{equation}
where the last line comes from the FKG inequality and the translation invariance of the process.

Collecting \eqref{firstboundG} and \eqref{GxiExi}, using that $\alpha'(p)$ and $\alpha_1(p)$ are uniformly bounded above and denoting by $C>0$, a constant that is uniform over all the variables involved, we get
\begin{equation}\label{GXI}
    \begin{split}
         \mathbb{P}\big(G^{\xi}\big) & \geq e^{-2\lambda C (1-p)N^{d+1}}\Big(1-e^{-CN} \Big)^{CN^d}\Big(1-\gamma/2-C_KN^d e^{-\ell_K N} \Big)\\
         &\times \Big(\mathbf{P}_{\zeta^{\Z^d}}^{\lambda p}\Big(\tau_{2Ne_1}\zeta_T^{\Z^d}\in H\Big)^2 - C_KN^{2d-1}e^{-\ell_K N}-\gamma/2\Big).
    \end{split}
\end{equation}
We are left to estimate the term $\mathbf{P}_{\zeta^{\Z^d}}^{\lambda p}\Big(\tau_{2Ne_1}\zeta_T^{\Z^d}\in H\Big)$. We claim that for $N$ and $K$ large enough,
\begin{equation}\label{CLAIM}
    \mathbf{P}^{\lambda p}\Big[\tau_{2N} \zeta_T^{\Z^d}\in H \Big] \geq 1-\gamma/4.
\end{equation}
Therefore, plugging this into \eqref{GXI}, by taking $N$ large enough, and then $p$ close enough to $1$, we get that \eqref{lowerboundprobagoodevent} holds, hence the result.

Let us prove the claim \eqref{CLAIM}. Since the law of $\zeta^{\Z^d}$ stochastically dominates the upper invariant measure $\overline{\nu}$ of the contact process with birth rate $\lambda p$, we have, by translation invariance of $\overline{\nu}$ and the relation \eqref{Dualrelationcontact}, 
\begin{equation}\label{Limitlastline}
    \begin{split}
        \mathbf{P}^{\lambda p}\Big[\tau_{2N} \zeta_T^{\Z^d}\in H \Big]&\geq \overline{\nu}\Big(\zeta,~ \mathbf{P}^{\lambda p}_{\zeta^{|[-K,K]^d}}\big(A(\zeta_t^{|[-K,K]^d})\neq \emptyset,~\forall t \geq 0 \big)>1-\gamma/2 \Big)\\
        &\geq \overline{\nu}\Big(\zeta,~ \overline{\nu}\big(\Tilde{\zeta},~A(\Tilde{\zeta})\cap A(\zeta^{|[-K,K]^d})\neq \emptyset\big)>1-\gamma/2\Big)\\
        & \underset{K \rightarrow \infty}{\longrightarrow}\overline{\nu}\Big(\zeta,~ \overline{\nu}\big(\Tilde{\zeta},~A(\Tilde{\zeta})\cap A(\zeta)\neq \emptyset\big)>1-\gamma/2\Big).
    \end{split}
\end{equation}
Now again by \eqref{Dualrelationcontact},
\begin{equation*}
\begin{split}
        \overline{\nu}\Big(\zeta,~ \overline{\nu}\big(\Tilde{\zeta},~A(\Tilde{\zeta})\cap A(\zeta)\neq \emptyset\big)>1-\gamma/2\Big) & = \overline{\nu}\Big(\zeta,~\mathbf{P}_{\zeta}^{\lambda p}\Big( A(\zeta_t)\neq \emptyset,~ \forall t \geq 0\Big)>1-\gamma/2\Big)\\
        &\geq \overline{\nu}\Big(\zeta,~\mathbf{P}_{\zeta}^{\lambda p}\Big( A(\zeta_t)\neq \emptyset,~ \forall t \geq 0\Big)=1\Big)\\
        &\geq \overline{\nu}\big(\zeta,~ A(\zeta)\neq \emptyset \big) = 1.
\end{split}
\end{equation*}
Therefore, the limit in the last line in \eqref{Limitlastline} equals $1$, so by taking $N$ and $K$ large enough, we get the desired claim. 

\begin{rem}
One could argue that the comparison theorem should be directly applicable to the $IS(\lambda,p)$ process, without having to use the intermediate $Spont(\lambda,p)$ process. This might be the case, using some more convoluted block construction machinery. However, for the block construction we suggest, this fails, since we strongly rely on the fact that the $Spont$ process restricted to a block, is smaller than the whole process (Lemma \ref{Restriction}). This does not hold for the $IS(\lambda,p)$ process since the presence of $1$'s just outside $R$ can lead to the birth of $-1$'s inside it, contrary to the $IS(\lambda,p)$ process restricted to exponential clocks ringing only inside $R$.
\end{rem}

\appendix
\section{Necessary and sufficient conditions for monotonicity}\label{ApppendixBorrello}
Here, we recall the result established by Davide Borrello in \cite{Borrello}, which characterizes the monotonicity of a particle system in terms of conditions on its transition rates. For simplicity, we provide a statement adapted to our model and refer to Theorem 2.4 of \cite{Borrello}, for a complete statement. Recall Borrello's notation for the birth and death rates of a particle system with state space $\Omega=\{-1,0,1\}^{\Z^d}$: for $\alpha,\beta \in \{-1,0,1\}$, and $0\leq k \leq 2$,
\begin{itemize}
    \item $R_{\alpha,\beta}^{0,k}$ is the rate at which site $y$ goes from state $\beta$ to $\beta + k$, due to the fact that $x$ is in state $\alpha$.
    \item $P_{\beta}^k$ is the rate at which site $y$ goes from state $\beta$ to $\beta + k$, independently of everything else.
    \item $R_{\alpha,\beta}^{-k,0}$ is the rate at which site $x$ goes from state $\alpha$ to $\alpha-k$, due to the fact that $y$ is in state $\beta$.
    \item $P_{\beta}^{-k}$ is the rate at which site $y$ goes from state $\beta$ to $\beta -k$, independently of everything else.
\end{itemize}
Introduce
\begin{equation*}
    \Pi_{\alpha,\beta}^{0,k}:= R_{\alpha,\beta}^{0,k} + P_{\beta}^k, ~~\text{and}~~\Pi_{\alpha,\beta}^{-k,0}:= R_{\alpha,\beta}^{-k,0} + P_{\alpha}^{-k}.
\end{equation*}
\begin{thm}{\cite[Theorem 2.4]{Borrello}} 
A particle system with transition rates $R_{\alpha,\beta}^{0,k}, P_{\beta}^k, R_{\alpha,\beta}^{-k,0}, P_{\beta}^{-k}$  is larger than a particle system with transition rates $\Tilde{R}_{\alpha,\beta}^{0,k}, \Tilde{P}_{\beta}^k, \Tilde{R}_{\alpha,\beta}^{-k,0}, \Tilde{P}_{\beta}^{-k}$, if and only if:
\begin{equation} \label{I1}
    \sum_{0\leq k \leq 2,~ k>j+\delta-\beta} \Tilde{\Pi}_{\alpha,\beta}^{0,k}  \leq \sum_{0\leq k \leq 2,~ k>j} \Pi_{\gamma,\delta}^{0,k}, 
\end{equation}
and
\begin{equation} \label{I2}
    \sum_{0\leq k \leq 2,~ k>h+ \gamma -\alpha} \Pi_{\gamma,\delta}^{-k,0}\leq \sum_{0\leq k \leq 2,~ k>h} \Tilde{\Pi}_{\alpha,\beta}^{-k,0} 
\end{equation}
for all $h,j\geq 0$, $(\alpha,\beta)$ and $(\gamma,\delta)$, such that $\alpha \leq \gamma$, and $\beta\leq \delta$.
\end{thm}\medskip

\noindent \textbf{Acknowledgements:} I am very greatful to Rinaldo Schinazi for suggesting this model to me. I also wish to thank Ellen Saada, Thomas Mountford and Assaf Shapira for some very useful discussions.

\bibliographystyle{plain}
\bibliography{biblio.bib}

\end{document}